\RequirePackage{fix-cm}

\documentclass[smallextended]{svjour3}       

\smartqed  
\usepackage{graphicx}
\usepackage{amsfonts}
\usepackage{amsmath}

\usepackage{float}
\usepackage{stfloats}

\begin{document}

\title{Closed-form expressions for Farhi's constant and related integrals and and its generalization}

\titlerunning{Closed-form expressions for Farhi's constant...}

\author{F\'{a}bio M.~.S.~Lima}

\institute{F.M.S. Lima \at
              Institute of Physics, University of Bras\'{i}lia, P.O.~Box 04455, Bras\'{i}lia, DF, 70919-970, Brazil. \\
              Tel.: +55 (061) 31076088\\
              \email{fabio@fis.unb.br}
}

\date{Received: 22 May 2019 / Accepted: date}

\maketitle

\begin{abstract}
In a recent work, Farhi developed a Fourier series expansion for the function $\,\ln{\Gamma(x)}\,$ on the interval $(0,1)$, which allowed him to derive a nice formula for the constant $\,\eta := 2 \int_0^1{\ln{\Gamma(x)} \, \sin{(2 \pi x)} \, dx}$. At the end of that paper, he asks whether $\eta$ could be written in terms of other known mathematical constants. Here in this work, after deriving a simple closed-form expression for $\eta$, I show how it can be used for evaluating other related integrals, as well as certain logarithmic series, which allows for a generalization in the form of a continuous function $\eta(x)$, $x \in [0,1]$. Finally, from the Fourier series expansion of $\,\ln{\Gamma(x)}$, $x \in (0,1)$, I make use of Parseval's theorem to derive a closed-form expression for $\,\int_0^1{\ln^2{\Gamma(x)}~dx}$.
\keywords{Special functions \and Log-gamma function \and Fourier series \and Series and summability}
\subclass{11Y60 \and 33B15 \and 42A16 \and 35C05}
\end{abstract}

\section{Introduction}

The Fourier series expansion of the real function $\,\ln{\Gamma(x)}\,$ over the open interval $(0,1)$, where $\,\Gamma(x) := \int_0^\infty{t^{\,x -1}\,e^{-t} \: d t}\,$ is the classical Gamma function, was developed by Farhi in a recent note~\cite{Farhi}. There, he shows that
\begin{equation}
\ln{\Gamma(x)} = \frac12 \ln{\pi} + \pi \, \eta \left(\frac12 -x \right) -\frac12 \ln{\sin{\!(\pi x)}} +\frac{1}{\pi} \sum_{n=1}^\infty{\frac{\ln{n}}{n} \, \sin{\!(2 \pi n x)}}
\label{eq:Farhi1}
\end{equation}
holds for all $\,x \in (0, 1)$, where
\begin{equation}
\eta := 2 \int_0^1{\ln{\Gamma(x)} \, \sin{\!(2 \pi x)} \: dx}
\label{eq:eta0}
\end{equation}
is the \emph{Farhi constant}. Numerical integration yields $\,\eta = 0.76874789\ldots$  At the end of Ref.~\cite{Farhi}, Farhi asks if $\,\eta\,$ could be written in terms of other known constants. Here in this paper, we make use of Farhi's formula itself to derive a closed-form expression for $\eta$ in terms of $\pi$, $\gamma$, and $\ln{(2 \pi)}$ only, where $\,\gamma := \lim_{\,n \rightarrow \infty}{(H_n -\ln{n})}\,$ is the Euler--Mascheroni constant, $\,H_n := \sum_{k=1}^n{\frac{1}{k}}\,$ being the $n$-th harmonic number. We also show how to generalize Eq.~\eqref{eq:eta0} and how Eq.~\eqref{eq:Farhi1} can be taken into account for evaluating some logarithmic series and other related integrals, e.g. $\,\int_0^{1/2}{\ln{\Gamma(x)} \, \sin{(2 \pi x)} \, d x}$, $\int_0^1{\ln{\Gamma(x)} \, \sin{(\pi x)} \, d x}$, $\int_0^1{\psi(x) \, \sin^2{\!(\pi x)} \, dx}$, $\int_0^1{\ln^2{\Gamma(x)} \, d x}$, and $\,\int_0^{1/2}{\psi(x) \, \sin^2{\!(\pi x)} \, d x}$, where $\,\psi(x) := \frac{d }{dx} \ln{\Gamma(x)}\,$ is the digamma function.

\section{Farhi's constant $\eta$ and some related integrals and series}

\begin{lemma}[Farhi's constant]
\label{lema:eta}
The exact closed-form
\begin{eqnarray*}
\eta = \frac{\:\gamma +\ln{(2 \pi)}}{\pi}
\end{eqnarray*}
holds, where $\,\eta\,$ is the definite integral in Eq.~\eqref{eq:eta0}.
\end{lemma}

\begin{proof}
On taking $\,x=\frac14\,$ in Farhi's formula, our Eq.~\eqref{eq:Farhi1}, one finds
\begin{equation}
\ln{\Gamma\!\left(\frac14 \right)} = \frac12 \ln{\pi} + \frac14 \pi \, \eta +\frac14 \ln{2} +\frac{1}{\pi} \,S \, ,
\label{eq:inicio}
\end{equation}
where
\begin{eqnarray}
S := \sum_{n=1}^\infty{\frac{\,\ln{n}}{n} \, \sin{\!\left(\frac{\pi}{2} \, n \right)}} = \sum_{m=1}^\infty{(-1)^{m+1} \: \frac{\,\ln{(2m -1)}}{2m -1}} \, ,
\end{eqnarray}
which agrees with Corollary~3 of Ref.~\cite{Farhi}. Now, we take into account a functional relation established by Coffey in Eq.~(3.13b) of Ref.~\cite{Coffey2011}, namely\footnote{We correct here a mistake in the Stieltjes constants at the left-hand side of the equation corresponding to this one, as it appears in Ref.~\cite{Coffey2011}.}
\begin{equation}
\gamma_1\!\left( \frac{a +1}{2} \right) -\gamma_1\!\left( \frac{a}{2} \right) = \ln{2} \, \left[ \psi\!\left( \frac{a+1}{2} \right) -\psi\!\left( \frac{a}{2} \right) \right] +2\,\sum_{k=0}^\infty{(-1)^{k+1}\:\frac{\ln{(k+a)}}{k+a}} \, ,
\label{eq:zeta}
\end{equation}
where $\,\gamma_1(a)\,$ is the first generalized Stieltjes constant, a coefficient in the Laurent series expansion of the Hurwitz zeta function $\,\zeta{(s,a)}\,$ about its simple pole at $\,s=1$,\footnote{Here, $\zeta{(s,a)} := \sum_{n=0}^\infty{1/(n +a)^s}$, $a \ne 0, -1, -2, \ldots$, a series that converges for all complex $s$ with $\,\Re{(s)} > 1$.}  i.e. $\frac{1}{s-1} +\sum_{n=0}^\infty{\frac{(-1)^n}{n!} \, \gamma_n(a)\,(s-1)^n}$.
For $\,a=\frac12$, one has
\begin{equation}
S = \frac14 \left[ \gamma_1\!\left( \frac14 \right) -\gamma_1\!\left( \frac34 \right) +\pi\,\ln{4} \right] ,
\label{eq:S}
\end{equation}
On the other hand, in Eq.~(3.21) of Ref.~\cite{Coffey2011} Coffey showed that
\begin{eqnarray}
\gamma_1\!\left( \frac14 \right) -\gamma_1\!\left( \frac34 \right) = -\,\pi \left[ \ln{(8 \pi)} +\gamma -2\,\ln{\!\left( \frac{\Gamma(1/4)}{\Gamma(3/4)} \right)}\right] \nonumber \\
= \pi \left[4 \,\ln{\Gamma\!\left( \frac14 \right)} -2\,\ln{4} -3\,\ln{\pi} -\gamma \right] ,
\label{eq:cof}
\end{eqnarray}
where the last step demanded the use of the reflection formula for $\Gamma(x)$, namely
\begin{equation}
\Gamma(1-x) \; \Gamma(x) = \frac{\pi}{\,\sin{(\pi x)}} \, .
\label{eq:gamareflex}
\end{equation}
On substituting the result of Eq.~\eqref{eq:cof} in Eq.~\eqref{eq:S}, one finds
\begin{equation}
S = \pi \, \ln{\Gamma\!\left( \frac14 \right)} -\frac{\pi}{2} \ln{2} -\frac34 \,\pi \ln{\pi} -\frac{\pi}{4} \, \gamma \, .
\end{equation}
The proof completes by substituting this expression in Eq.~\eqref{eq:inicio}. \qed
\end{proof}

The closed-form expression for $\,\eta\,$ established above could also have been deduced from Eq.~(6.443.1) of Ref.~\cite{Gradsh} (the case $\,n=1$), or it could have been determined by comparing Eq.~\eqref{eq:Farhi1} with the Kummer's Fourier series for $\,\ln{\Gamma(x)}\,$ mentioned by Connon in Eq.~(2.9) of Ref.~\cite{Connon2009} (see also Eq.~(5.8) of Ref.~\cite{Connon2012}), namely
\begin{equation}
\ln{\Gamma(x)} = \frac12 \ln{\!\left(\frac{\pi}{\,\sin{(\pi x)}}\right)} +\left(\frac12 - x \right) [\gamma +\ln{(2 \pi)}] +\frac{1}{\pi} \sum_{n=1}^\infty{\frac{\ln{n}}{n} \, \sin{(2 \pi n x)}} \, .
\label{eq:Connon1}
\end{equation}
Note that in Corollary~9.6.53 of Ref.~\cite{bk:Cohen} Cohen uses Abel summation to derive a formula similar to that by Connon which holds for all real $\,x \not \in \mathbb{Z}$. Another generalization of Farhi's formula was proposed by Blagouchine in Ex.~20(c) of Ref.~\cite{Blago2014}, namely
\begin{eqnarray}
\sum_{n=1}^\infty{\frac{\,\ln{(b\,n)}}{n} \, \sin{(n\,\phi)}} = \pi \, \ln{\Gamma\!\left(\frac{\phi}{2 \pi} \right)} +\frac{\pi}{2} \, \ln{\sin{\!\left(\frac{\,\phi}{2}\right)}} -\frac{\pi}{2} \, \ln{\pi} \nonumber \\
+ \, \frac{\:\phi -\pi}{2} \left[ \gamma +\ln{\!\left(\!\frac{\,2 \pi}{b}\right)} \right] ,
\label{eq:Blag1}
\end{eqnarray}
which is valid for all $\,b>0\,$ and $\,\phi \in (0, 2 \pi)$. On substituting $\,\phi = 2 \pi x$, one shows that Connon's formula is the particular case $\,b=1\,$ of Eq.~\eqref{eq:Blag1}.

Since the Farhi constant $\eta$ is defined as an integral, other similar integrals can be explored, as, for instance,

\begin{theorem}[An integral involving $\,\psi(x)\,$]
\label{teo:digamma1}
The following exact closed-form result holds:
\begin{eqnarray*}
\int_0^1{\psi(x) \: \sin^2{\!(\pi x)} \: dx} = -\,\frac{\:\gamma +\ln{(2 \pi)}\,}{2} \, .
\end{eqnarray*}
\end{theorem}

\begin{proof}
From Lemma~\ref{lema:eta}, one has
\begin{equation}
\pi \, \eta = 2 \pi \int_0^1{\ln{\Gamma(x)} \left[ 2 \sin{(\pi x)} \, \cos{(\pi x)} \right] \, d x} = \gamma +\ln{(2 \pi)} \, ,
\end{equation}
which implies that
\begin{equation}
I := \int_0^1{\ln{\Gamma(x)} \, \sin{(\pi x)} \, \cos{(\pi x)} \, d x} = \frac{\gamma +\ln{(2 \pi)}}{4 \pi} \, .
\label{eq:defI}
\end{equation}
Integration by parts then yields
\begin{eqnarray}
I = \left[ \ln{\Gamma(x)} \, \frac{\,\sin^2{(\pi x)}}{\pi} \right]_{0^{+}}^1 - \int_0^1{\frac{\,\sin{(\pi x)}}{\pi} \, \left[ \psi(x) \, \sin{(\pi x)} +\pi\,\cos{(\pi x)}\,\ln{\Gamma(x)} \right] d x} \nonumber \\
= \, 0 -\frac{1}{\pi} \int_0^1{\psi(x) \, \sin^2{(\pi x)} \, d x} -\int_0^1{\sin{(\pi x)} \, \cos{(\pi x)} \, \ln{\Gamma(x)} \: d x} \nonumber \\
= - \frac{1}{\pi} \int_0^1{\psi(x) \, \sin^2{(\pi x)} \, d x} -I \, . \qquad
\end{eqnarray}
Therefore
\begin{equation}
2 \, I = -\frac{1}{\,\pi} \int_0^1{\psi(x) \, \sin^2{(\pi x)} \: d x} \, .
\end{equation}
The final result follows by substituting this integral in Eq.~\eqref{eq:defI}. \qed
\end{proof}

A simple closed-form expression can also be determined for a similar integral obtained by halving the argument of the sine in Eq.~\eqref{eq:eta0}, i.e.
\begin{equation}
\int_0^1{\ln{\Gamma(x)} \, \sin{\!(\pi x)} \: dx} = 0.46205312 \ldots
\label{eq:kapa}
\end{equation}

\begin{theorem}[An integral involving $\,\ln{\Gamma{(x)}}\,$]
\label{teo:kapa}
The exact closed-form result
\begin{eqnarray*}
\int_0^1{\ln{\Gamma(x)} \, \sin{\!(\pi x)} \: dx} = \frac{\:\ln{\pi} -\ln{2} +1}{\pi}
\end{eqnarray*}
holds.
\end{theorem}

\begin{proof}
On taking into account the logarithmic form of Eq.~\eqref{eq:gamareflex}, namely
\begin{equation}
\ln{\Gamma(1-x)} +\ln{\Gamma(x)} = \ln{\pi} -\ln{\sin{(\pi x)}} \, ,
\label{eq:lngamareflex}
\end{equation}
one finds
\begin{eqnarray}
\int_0^1{\left[ \, \ln{\Gamma(1-x)} +\ln{\Gamma(x)} \right] \sin{(\pi x)} \: d x} \nonumber \\
= \int_0^1{\left[ \, \ln{\pi} -\ln{\sin{(\pi x)}}\right] \sin{(\pi x)} \: d x} \nonumber \\
= \ln{\pi} \int_0^1{\sin{(\pi x)} \: d x} -\int_0^1{\ln{\sin{(\pi x)}} \, \sin{(\pi x)} \: d x} \nonumber \\
= \ln{\pi}~\frac{\,2}{\,\pi} -\frac{1}{\pi} \int_0^\pi{\sin{\tilde{\theta}} \; \ln{\sin{\tilde{\theta}}} \: d \tilde{\theta}} \, .
\end{eqnarray}
The symmetries of the first integrand  with respect to $\,x = \frac12\,$ and of $\,\sin{\tilde{\theta}}\,$ with respect to $\,\tilde{\theta} = \pi/2\,$ lead to
\begin{eqnarray}
2 \int_0^1{\ln{\Gamma(x)} \, \sin{\!(\pi x)} \: dx} = \frac{\,2}{\,\pi} \, \ln{\pi} -\frac{2}{\pi} \int_0^{\,\pi/2}{\!\sin{\tilde{\theta}} \: \ln{\sin{\tilde{\theta}}} \; d \tilde{\theta}} \nonumber \\
= 2 \, \frac{\,\ln{\pi}}{\pi} -2\,\frac{\:\ln{2} -1}{\pi} \, .
\end{eqnarray}
\qed
\end{proof}

On searching for other closed-form results, I have realized that the integration of both sides of Farhi's formula, our Eq.~\eqref{eq:Farhi1}, could lead to an interesting result. As usual, $\,\mathrm{Cl}_2(\theta) := \Im{\left[ \mathrm{Li}_2\left(e^{\,i \, \theta}\right) \right]} = \sum_{n=1}^\infty{\sin{(n \, \theta)}/n^2}\,$ is the Clausen function and $\,\zeta'(s,a)\,$ denotes a partial derivative with respect to $s$.

\begin{theorem}[Integration of Farhi's formula]
\label{teo:int}
For all $\,x \in (0 ,1)$,\footnote{In the interval $(0,1)$, this logarithmic series corresponds to $\, - \, \frac12 \: \frac{\partial }{\partial s} \left[ \mathrm{Li}_s\!\left(e^{2 \pi x \,i} \right) +\mathrm{Li}_s\!\left(e^{-2 \pi x \,i} \right) \right]_{s=2}$, a result similar to that in Eq.~(1.18) of Ref.~\cite{Moll}.}
\begin{eqnarray*}
\sum_{n=1}^\infty{\frac{\,\ln{n}}{n^2} \, \cos{(2 \pi n \,x)}} = \pi^2 \left( x\,(1 -x) - \frac16 \right) \left[\,\gamma +\ln{(2 \pi)} -1 \right] +\frac{\pi}{2}\,\mathrm{Cl}_2(2 \pi x) \nonumber \\
-\,2 \pi^2\,\zeta'(-1, x) \, .
\end{eqnarray*}
\end{theorem}

\begin{proof}
The integration of both sides of Eq.~\eqref{eq:Farhi1} from $0$ to $x$, for any $\,x \in (0,1)$, yields
\begin{eqnarray}
\int_0^x{\ln{\Gamma(t)} \, d t} = \frac{\ln{\pi}}{2} \, x +\frac{\pi \, \eta}{2}\,x -\frac{\pi \,\eta}{2} \,x^2 -\frac12 \int_0^x{\ln{\sin{(\pi t)}} \, d t} \nonumber \\
+\,\frac{1}{\pi} \int_0^x{ \left[\,\sum_{n=1}^\infty{\frac{\ln{n}}{n} \sin{(2 \pi n \,t)} } \right] d t} \nonumber \\
= \left[\frac{\ln{\pi}}{2} +\frac{\pi \eta}{2} \right] x -\frac{\pi \eta}{2} \,x^2 -\frac12 \int_0^x{\ln{\sin{(\pi t)}} \, d t} \nonumber \\
+\,\frac{1}{\pi} \sum_{n=1}^\infty{\frac{\ln{n}}{n} \int_0^x{\sin{(2 \pi n \,t)} \, d t}} \, ,
\label{eq:ini0}
\end{eqnarray}
where $\,\pi\,\eta = \gamma +\ln{(2 \pi)}$, as proved in Lemma~\ref{lema:eta}. The absolute convergence of the last series above, for all $\,x \in (0,1)$, is sufficient for the application of Fubini's theorem --- see, e.g., the corollary of Theorem~7.16 in Ref.~\cite{bk:Rudin} ---, which validates the interchange of the integral and the series in the last step. Now, let us solve each integral above separately. Firstly, by definition, $\int_0^x{\ln{\Gamma(t)} \, d t} = \psi^{(-2)}(x)$, known as the \emph{negapolygamma} function, for which Adamchik showed in Ref.~\cite{Adamchik}, as a particular case of his Eq.~(14), that
\begin{equation}
\psi^{(-2)}(x) = \frac{x \, (1 -x)}{2} +\frac{x}{2}\,\ln{(2 \pi)} -\zeta'(-1) +\zeta'(-1, x) \, ,
\label{eq:negapol2}
\end{equation}
where $\,\zeta'(-1) = 1/12 -\ln{A}$, $\,A\,$ being the Glaisher-Kinkelin constant.\footnote{In virtue of Eq.~(5.2) of Ref.~\cite{Connon2009}, namely $\,\zeta'(-1) -\zeta'(-1,t) = \ln{G(1+t)} -t\,\ln{\Gamma(t)}$, it is possible to write Eq.~\eqref{eq:negapol2} in terms of the Barnes $G$-function, but we shall not explore this form here.} Secondly,
\begin{equation}
\int_0^{\,x}{\ln{\sin{(\pi t)}} \, d t} = -\,\frac{\,\mathrm{Cl}_2(2 \pi x)\,}{2 \pi} - x \, \ln{2} \, .
\end{equation}
Finally, on substituting $\,\int_0^x{\sin{(2 \pi n \,t)} \, d t} = -\,[\cos{(2 \pi n \,x)} -1]/(2 \pi n)\,$ in the last series in Eq.~\eqref{eq:ini0}, one finds
\begin{eqnarray}
\sum_{n=1}^\infty{\frac{\ln{n}}{n} \int_0^x{\sin{(2 \pi n \,t)} \, d t}} = -\frac{1}{2 \pi} \, \sum_{n=1}^\infty{\frac{\ln{n}}{n^2} \, [\cos{(2 \pi n \,x)}-1] } \nonumber \\
= -\frac{1}{\,2 \pi} \left[ \, \sum_{n=1}^\infty{\frac{\ln{n}}{n^2} \, \cos{(2 \pi n \,x)}} -\sum_{n=1}^\infty{\frac{\ln{n}}{n^2}} \right] ,
\end{eqnarray}
which can be further simplified by noting that
\begin{equation}
\sum_{n=1}^\infty{\frac{\ln{n}}{n^2}} = -\,\zeta'(2) = \frac{\,\pi^2}{6} \left[ 12 \, \ln{A} -\gamma -\ln{(2 \pi)} \right] ,
\label{eq:somaln}
\end{equation}
as shown by Glaisher in 1894~\cite{Glaisher}. The proof completes by substituting the four closed-form expressions above there in Eq.~\eqref{eq:ini0}. \qed
\end{proof}

\begin{remark}
On putting $\,x = \frac12\,$ (or $\,x = \frac14$) in Theorem~\ref{teo:int}, one finds
\begin{eqnarray*}
\sum_{n=1}^\infty{(-1)^n \, \frac{\,\ln{n}}{n^2}} = \frac{\,\pi^2}{12} \, \left[\,\gamma +\ln{(2 \pi)} -1 \right] -2 \pi^2 \, \zeta'\!\left(-1, \frac12 \right) \nonumber \\
= \frac{\:\pi^2}{12} \left[\,\gamma +\ln{(4 \pi)} -12\,\ln A \,\right] .
\end{eqnarray*}
\end{remark}

Since the function $\,\cos{(2 \pi n\,x)}$, $x \in (0,1)$, remains the same when we exchange $\,x\,$ by $\,1-x$, then both the first term in the right-hand side and the series in Theorem~\ref{teo:int} also do, whereas $\,\mathrm{Cl}_2(2 \pi x)\,$ changes the sign. This simple observation leads to

\begin{corollary}[Reflection formula for $\,\zeta'{(-1,x)}\,$]
For all $\,x \in [0,1]$,
\begin{eqnarray*}
\zeta'{(-1, x)} -\zeta'{(-1, 1-x)} = \frac{\,\mathrm{Cl}_2(2 \pi x)}{2 \pi} \, .
\end{eqnarray*}
\end{corollary}

This formula corresponds to the case $\,n=1\,$ of Eq.~(21) in Ref.~\cite{Miller}.  Note that it remains valid at both endpoints $\,x=0\,$ and $\,x=1\,$ because $\,\zeta'(-1,0) = \zeta'(-1,1) = \zeta'(-1)\,$ and $\,\mathrm{Cl}_2(0) = \mathrm{Cl}_2(2 \pi) = 0$.

\begin{remark}
For instance, for $\,x=\frac14$ the reflection formula yields
\begin{eqnarray*}
\zeta'{\left(-1, \frac14 \right)} -\zeta'{\left(-1, \frac34 \right)} = \frac{\mathrm{Cl}_2(\pi/2)}{2 \pi} = \frac{G}{2 \pi} \, ,
\end{eqnarray*}
where $\,G := \sum_{n=0}^\infty{(-1)^n/(2n +1)^2}\,$ is Catalan's constant.
\end{remark}

The presence of the factor $\,\sin{(2 \pi x)}\,$ accompanying the log-Gamma function in Eq.~\eqref{eq:eta0} suggests that further results can be found by multiplying both sides of Farhi's formula by that factor before integration. Then, let us define the real function
\begin{equation}
\eta(x) := 2 \int_0^{\,x}{\ln{\Gamma(t)}\,\sin{(2 \pi t)} \: d t} \, , \quad x \in (0,1) \, ,
\label{eq:etax0}
\end{equation}
as a generalization of Farhi's constant $\eta$.

\begin{theorem}[Another integral of Farhi's formula]
\label{teo:int2}
For all $\,x \in (0 ,1)$, one has
\begin{eqnarray*}
\eta(x) = \frac{\,\gamma +\ln{(2 \pi)}}{\pi} \left[ \sin^2{\!\left(\frac{\theta}{2}\right)} +\frac{\theta\,\cos{\theta} -\sin{\theta}}{2 \pi} \right] +\frac{\,\sin^2{\!(\theta/2)}}{2 \pi} \left[1 +2 \, \ln{\!\left(\frac{\pi}{\sin{(\theta/2)}}\right)} \right] \nonumber \\
+\,\frac{\,\cos{\theta}}{\pi^2} \, \sum_{n=2}^\infty{\frac{\ln{n}}{\,n \, (n^2-1)} \, \sin{(n \theta)}} \,-\frac{\,\sin{\theta}}{\pi^2} \, \sum_{n=2}^\infty{\frac{\ln{n}}{\: n^2 -1} \, \cos{(n \theta)}} \, ,
\end{eqnarray*}
where $\,\theta = 2 \pi x$.
\end{theorem}

\begin{proof}
The multiplication of both sides of Eq.~\eqref{eq:Farhi1} by $\,\sin{(2 \pi t)}$, followed by integration from $0$ to $x$, for any $\,x \in (0,1)$, yields
\begin{eqnarray}
\int_0^{\,x}{\ln{\Gamma(t)} \, \sin{(2 \pi t)}\, d t} \nonumber \\
= \frac{\,\ln{\pi}}{4 \pi} \int_0^{\,\theta}{\sin{\tilde{\theta}} \, d \tilde{\theta}} +\frac{\gamma +\ln{(2 \pi)}}{4 \pi} \int_0^{\,\theta}{\!\left(1 -\frac{\tilde{\theta}}{\pi} \right) \sin{\tilde{\theta}} \, d \tilde{\theta}} \nonumber \\
-\,\frac{1}{\,2 \pi} \int_0^{\,\theta}{\ln{\!\left[\sin{\!\left(\frac{\tilde{\theta}}{2}\right)}\right]} \sin{\!\left( \frac{\tilde{\theta}}{2} \right)} \cos{\!\left( \frac{\tilde{\theta}}{2} \right)} \, d \tilde{\theta}} \nonumber \\
+ \, \frac{1}{\,2 \pi^2} \sum_{n=2}^\infty{\,\frac{\,\ln{n}}{n} \, \int_0^\theta{\sin{\!\left(n \tilde{\theta}\right)} \, \sin{\tilde{\theta}} \: d \tilde{\theta}}} \, ,
\label{eq:ini1}
\end{eqnarray}
where $\,\theta = 2 \pi x\,$ and $\tilde{\theta}$ is just a dummy variable. On applying the trigonometric identity $\,\sin{\alpha}\,\sin{\beta} = \frac12 \, [\cos{(\alpha -\beta)} -\cos(\alpha +\beta)]\,$ to the last integral above, one finds
\begin{eqnarray}
2 \int_0^{\,x}{\ln{\Gamma(t)} \, \sin{(2 \pi t)}\, d t} = \frac{\ln{\pi}}{2 \pi} \, (1 -\cos{\theta}) \nonumber \\
+\frac{\gamma +\ln{(2 \pi)}}{2 \pi} \left(\!1 -\cos{\theta} -\frac{1}{\pi} \int_0^{\,\theta}{\!\tilde{\theta}\,\sin{\tilde{\theta}} \, d \tilde{\theta}} \right) -\,\frac{2}{\pi} \int_0^{\,b}{u \, \ln{u} \: d u} \nonumber \\
+\,\frac{1}{\,2 \pi^2} \sum_{n=2}^\infty{\frac{\,\ln{n}}{n} \int_0^{\,\theta}{\!\!\left\{ \cos{\!\left[(n-1) \,\tilde{\theta}\right]} -\cos{\!\left[(n+1) \,\tilde{\theta}\right]} \right\} d \tilde{\theta}}} \, ,
\label{eq:2}
\end{eqnarray}
where $\,b = \sin{(\theta/2)}$. The remaining integrals can be solved in terms of elementary functions:
\begin{eqnarray}
\int_0^\theta{\tilde{\theta} \, \sin{\tilde{\theta}} \, d \tilde{\theta}} = \sin{\theta} -\theta \, \cos{\theta} \, , \\
\int_0^b{u \, \ln{u} \, d u} = \frac{\,b^2}{4} \: (2 \ln{b} -1) \, ,
\end{eqnarray}
and
\begin{eqnarray}
\int_0^{\,\theta}{\!\!\left\{\,\cos{\left[(n-1) \,\tilde{\theta}\right]} -\cos{\left[(n+1) \,\tilde{\theta}\right]} \right\} d \tilde{\theta}} = \frac{\,\sin{[(n-1)\,\theta]}}{n-1} - \frac{\,\sin{[(n+1)\,\theta]}}{n+1} \nonumber \\
= \frac{\,(n+1)\,[\sin{(n \theta)}\,\cos{\theta} -\sin{\theta}\,\cos{(n \theta)}] -(n-1)\,[\sin{(n \theta)}\,\cos{\theta} +\sin{\theta}\,\cos{(n \theta)}]}{n^2-1} \nonumber \\
= 2 \, \frac{\,\sin{(n \theta)}\,\cos{\theta} -n\,\sin{\theta}\,\cos{(n \theta)}}{n^2-1} \, , \qquad
\end{eqnarray}
The substitution of the last three expressions, above, there in Eq.~\eqref{eq:2} completes the proof. \qed
\end{proof}

\begin{remark}
In particular, for $\,x = 1/2\,$ our Theorem~\ref{teo:int2} promptly yields
\begin{equation}
\eta\!\left( \frac12 \right) = 2 \int_0^{\,1/2}{\ln{\Gamma(t)} \, \sin{(2 \pi t)} \: dt} = \frac{\,\gamma +\ln{(2 \pi)} +2 \ln{\pi} +1}{2 \pi} \, .
\label{eq:eta12}
\end{equation}
\end{remark}

The analytic expression of $\,\eta(x)\,$ established in Theorem~\ref{teo:int2} defines a continuous real function for all $\, x \in (0,1)$. However, it involves a logarithmic term which is undefined at the endpoints. This problem can be fixed by \emph{defining} $\,\eta(0) = \lim_{x \rightarrow 0^{+}}{\eta(x)}\,$ and $\,\eta(1) = \lim_{x \rightarrow 1^{-}}{\eta(x)}$, in a manner to make $\,\eta(x)\,$ continuous for all $\,x \in [0, 1]$.

\begin{theorem}[Extending the domain of $\eta(x)\,$]
\label{teo:extension}
The domain over which the real function $\,\eta(x) = 2 \, \int_0^x{\ln{\Gamma{(t)}} \, \sin{(2 \pi t)}~dt}\,$ is continuous can be extended to the closed interval $[0,1]$ by defining $\,\eta(0) := 0\,$ and $\,\eta(1) := \eta = [\gamma +\ln{(2 \pi)}]/\pi$.
\end{theorem}

\begin{proof}
In the analytic expression established for $\,\eta(x)\,$ in Theorem~\ref{teo:int2}, all terms promptly nullify at $\,x=0$, except the one involving $\,\ln{[\pi/\sin{(\theta/2)}]}$, where $\,\theta = 2 \pi x$. The choice $\,\eta(0) = 0\,$ then comes from the limit
\begin{eqnarray}
\lim_{\theta \rightarrow 0^{+}}{\sin^2{\!\left( \frac{\theta}{2} \right)} \, \ln{\!\left[\frac{\pi}{\sin{(\theta/2)}}\right]}} = \ln{\pi} \lim_{\theta \rightarrow 0^{+}}{\sin^2{\!\left( \frac{\theta}{2} \right)}} - \lim_{\theta \rightarrow 0^{+}}{\sin^2{\!\left( \frac{\theta}{2} \right)}  \ln{\!\left[\sin{\!\left( \frac{\theta}{2} \right)}\right]}} \nonumber \\
= - \lim_{\alpha \rightarrow 0^{+}}{\sin^2{\!\alpha} \; \ln{\left(\sin{\alpha}\right)}} =  \lim_{\alpha \rightarrow 0^{+}}{\!\frac{\,\ln{( \csc{\alpha} )}}{\csc^2{\alpha} \, }} \nonumber \\
= \lim_{y \rightarrow +\infty}{\!\frac{\,\ln{y}\,}{y^2}} = \lim_{y \rightarrow \infty}{\!\frac{1}{\:2\,y^2}} = 0 \, , \qquad
\label{eq:lims}
\end{eqnarray}
where the L'H\^{o}pital rule was applied in the last step.

The appropriate choice for $\,\eta(1)\,$ is found by taking $\,x \rightarrow 1^{-}$ (or, equivalently, $\theta \rightarrow {2 \pi}^{-}$) in the expression of $\,\eta(x)\,$ in Theorem~\ref{teo:int2}, which yields
\begin{eqnarray}
\lim_{\; x \, \rightarrow 1^{-}}{\eta(x)} = \frac{\,\gamma +\ln{(2 \pi)}}{\pi} +\frac{1}{\,2 \pi} \! \lim_{\;\, \theta \, \rightarrow {\,2 \pi}^{-}}{\!\sin^2{\!\left(\frac{\theta}{2} \right)} \left[1 +2 \, \ln{\!\left(\frac{\pi}{\sin{(\theta/2)}}\right)} \right]} \nonumber \\
= \eta +\frac{1}{\,2 \pi} \! \lim_{\;\, \theta/2 \, \rightarrow {\,\pi}^{-}}{\sin^2{\!\left(\frac{\theta}{2} \right)}} +\frac{1}{\pi} \! \lim_{\;\, \theta/2 \, \rightarrow {\,\pi}^{-}}{\sin^2{\!\left(\frac{\theta}{2} \right)} \, \ln{\!\left(\frac{\pi}{\sin{(\theta/2)}}\right)} } \nonumber \\
= \eta +\frac{\,\ln{\pi}}{\pi} \! \lim_{\;\, \theta/2 \, \rightarrow {\,\pi}^{-}}{\sin^2{\!\left(\frac{\theta}{2} \right)}} -\frac{1}{\pi} \! \lim_{\;\, \theta/2 \, \rightarrow {\,\pi}^{-}}{\sin^2{\!\left(\frac{\theta}{2} \right)} \, \ln{\!\left[\sin{\left(\frac{\theta}{2} \right)} \right]} } \nonumber \\
= \eta -\frac{1}{\pi} \! \lim_{\; \alpha \, \rightarrow {\,\pi}^{-}}{\sin^2{\!\alpha} \; \ln{\left(\sin{\alpha}\right)}} \, . \qquad
\label{eq:lims2}
\end{eqnarray}
As shown in Eq.~\eqref{eq:lims}, the last limit above is null. \qed
\end{proof}

Note that our choices for $\,\eta(0)\,$ and $\,\eta(1)\,$ agree with the values obtained directly from the integral definition of $\,\eta(x)\,$ in Eq.~\eqref{eq:etax0}, namely $\,\eta(0) = 2 \int_0^0{\ln{\Gamma(t)} \: \sin{(2 \pi t)} \: dt} = 0\,$ and $\,\eta(1) = 2 \int_0^1{\ln{\Gamma(t)} \: \sin{(2 \pi t)} \: dt} = \eta$, as defined in Eq.~\eqref{eq:eta0}.

Interestingly, we can use the function $\,\eta(x)$, as defined in Eq.~\eqref{eq:etax0}, to generalize Theorem~\ref{teo:digamma1}.

\begin{theorem}[Generalization of the digamma integral]
\label{teo:intpsix}
For all $\,x \in (0,1]$, the closed-form result
\begin{eqnarray*}
\int_0^{\,x}{\psi(t) \: \sin^2{\!(\pi t)} \: dt} = \ln{\Gamma(x)} \: \sin^2{\!(\pi x)} -\frac{\pi}{2} \: \eta(x)
\end{eqnarray*}
holds, where $\,\eta(x)\,$ is the function defined in Eq.~\eqref{eq:etax0} for $\,x \in (0,1)$, a domain extended to $\,[0,1]\,$ in Theorem~\ref{teo:extension}.
\end{theorem}

\begin{proof}
From the integral definition of $\,\eta(x)\,$ in Eq.~\eqref{eq:etax0}, one has
\begin{equation}
\pi \, \eta(x) = 4 \pi \int_0^x{\ln{\Gamma(t)} \: \sin{(\pi t)} \, \cos{(\pi t)} \: d t} = 4 \int_0^{\,\pi x}{\ln{\Gamma\!\left( \frac{y}{\pi} \right)} \sin{y} \, \cos{y} \: d y} ,
\end{equation}
for all $\,x \in (0,1)$. Now, define the last integral, above, as $I$ and integrate it by parts, following the same steps as those in the proof of Theorem~\ref{teo:digamma1}. After some algebra, one finds
\begin{equation}
2\,I = \frac{\pi}{2} \: \eta(x) = \ln{\Gamma(x)} \, \sin^2(\pi x) - \frac{1}{\pi} \int_0^{\,\pi x}{\psi\!\left(\frac{y}{\pi}\right) \, \sin^2{y} \: d y} \, ,
\end{equation}
from which the closed-form expression for $\,\int_0^x{\psi(t) \: \sin^2{(\pi t)} \, dt}\,$ promptly follows. Finally, for $\,x=1\,$ this closed-form expression reduces to
\begin{equation}
\int_0^1{\psi(t) \: \sin^2{\!(\pi t)} \: dt} = 0 -\frac{\pi}{2} ~ \eta(1) = -\frac{\,\pi \, \eta\,}{2} \, ,
\end{equation}
which agrees with Theorem~\ref{teo:digamma1}. \qed
\end{proof}

For $\,x=\frac12$, on taking into account the special value we have found in Eq.~\eqref{eq:eta12}, one finds

\begin{corollary}
\begin{eqnarray*}
\int_0^{\,1/2}{\psi(t) \, \sin^2{\!(\pi t)} \: d t} = - \frac{\:\gamma +\ln{(2 \pi)} +1}{4} \, .
\end{eqnarray*}
\end{corollary}

Another useful generalization, from the point of view of Fourier series expansions, is obtained by inserting a positive integer parameter $\,k\,$ in the argument of the sine function in Eq.~\eqref{eq:eta0}, as follows:
\begin{equation}
\eta_k := 2 \int_0^1{\ln{\Gamma(t)} \: \sin{(2 \pi k \, t)} \: dt} \, .
\label{eq:etak}
\end{equation}
The problem of finding closed-form expressions for the Fourier coefficients of $\,\ln{\Gamma(x)}\,$ on the interval $\,(0,1)\,$ was solved in details by Farhi in Ref.~\cite{Farhi}, the result being
\begin{eqnarray}
a_0 = \int_0^1{\ln{\Gamma(t)} \: d t} = \frac12\,\ln{(2 \pi)} \, , \\
a_k = 2 \int_0^1{\ln{\Gamma(t)} \, \cos{(2 \pi k \, t)} \: d t} = \frac{1}{\,2 k} \: , \quad k \ge 1 \, , \label{eq:aks} \\
b_k = 2 \int_0^1{\ln{\Gamma(t)} \, \sin{(2 \pi k \, t)} \: d t} = \eta_k = \frac{\ln{k}}{\,\pi k} +\frac{\eta}{k} \: , \quad k \ge 1 \, . \label{eq:bks}
\label{eq:defcoefsFarhi}
\end{eqnarray}
Here, we are considering the Fourier series in the form
\begin{equation}
a_0 +\sum_{k=1}^\infty{\,a_k\,\cos{(2 \pi k x)}} +\sum_{k=1}^\infty{\,b_k\,\sin{(2 \pi k x)}} \: ,
\label{eq:defFS}
\end{equation}
as adopted in Ref.~\cite{Farhi}. There in that paper, it is shown that this series, with the coefficients given in Eq.~\eqref{eq:defcoefsFarhi}, converges to $\,\ln{\Gamma(x)}\,$ for all $\,x \in (0,1)$.

For $\,x=\frac12\,$ one readily finds $\,\ln{\Gamma(\frac12)} = \frac12 \, \ln{\pi} = \ln{\sqrt{\pi}}\,$, a well-known result. Other less-obvious results arise by considering distinct rational values of $x$. For instance, for $\,x=\frac13\,$ one finds
\begin{eqnarray}
\qquad \; \ln{\Gamma\!\left( \frac13 \right)} = \frac{\sqrt{3}}{6 \pi} \left[\gamma_1\!\left(\frac{1}{3}\right) -\gamma_1\!\left(\frac{2}{3}\right) \right] +\frac{\gamma }{6} +\frac23 \, \ln{(2 \pi)} -\frac{\ln{3}}{12} \, , \quad
\end{eqnarray}
which agrees with the closed-form expressions for $\gamma_1\!\left(\frac{1}{3}\right)$ and $\gamma_1\!\left(\frac{2}{3}\right)$ found by Blagouchine in Eq.~(61) of Ref.~\cite{Blago2014}. For $\,x=\frac14$, one finds
\begin{eqnarray}
\ln{\Gamma\!\left( \frac14 \right)} = \frac12 \, \sum_{k=1}^{\infty}{\frac{\,\cos{(k \pi/2)}}{k}} +\frac{\,\ln{(2 \pi)}}{2} +\frac{1}{\pi} \, \sum_{k=1}^{\infty}{\frac{\,\sin{(k \pi/2)}}{k} \, (\pi\,\eta +\ln{k})} \nonumber \\
= \frac{\,\ln{2}}{4} +\frac{\ln{(2 \pi)}}{2} +\frac{1}{4} \left[ \pi \,\eta +\frac{\gamma_1(1/4) -\gamma_1(3/4)}{\pi} \right] , \qquad
\end{eqnarray}
which simplifies just to our Eq.~\eqref{eq:cof}, so this evaluation can be viewed as an alternative derivation of Coffey's formula~\cite{Coffey2011}.

With the Fourier series expansion of $\,\ln{\Gamma(x)}$, $x \in (0,1)$, in hands, we can use Parseval's theorem to derive an additional closed-form result.

\begin{theorem}[Parseval's theorem for $\ln{\Gamma(x)}\,$]
\label{teo:Parseval}
\begin{eqnarray*}
\int_0^1{\ln^2{\Gamma(t)} \: d t} = 2 \ln{A}\:[\gamma +\ln{(2 \pi)}] -\frac{\,\gamma^2}{12} +\frac{\,\pi^2}{48} +\frac{\ln(2 \pi)}{6} \: [\,\ln{(2 \pi)} -\gamma\,] +\frac{\zeta''(2)}{2 \pi^2} \, .
\end{eqnarray*}
\end{theorem}

\begin{proof}
On applying Parseval's theorem to the Fourier series in Eq.~\eqref{eq:defFS}, one finds
\begin{eqnarray}
\int_0^1{\ln^2{\Gamma(t)}~dt} = a_0^2 +\frac12 \, \sum_{k=1}^\infty{a_k^{\,2}} +\frac12 \, \sum_{k=1}^\infty{b_k^{\,2}} \, ,
\end{eqnarray}
where the Fourier coefficients are those given in Eq.~\eqref{eq:defcoefsFarhi}. This expands to
\begin{eqnarray}
\int_0^1{\ln^2{\Gamma(t)}~dt} = \frac{\ln^2{(2 \pi)}}{4} +\frac18 \: \zeta{(2)} +\frac12 \, \sum_{k=1}^\infty{\left( \frac{\ln^2{k}}{\,\pi\,k^2} +\frac{\eta^2}{k^2} +\frac{2\,\eta}{\pi}~\frac{\ln{k}}{k^2}\right)} \nonumber \\
= \frac{\ln^2{(2 \pi)}}{4} +\frac{\pi^2}{48} +\frac{1}{2\,\pi} \, \sum_{k=2}^\infty{\frac{\ln^2{k}}{k^2}} +\frac{\eta^2}{2}~\frac{\pi^2}{6} +\frac{\eta}{\pi}\,\sum_{k=2}^\infty{\frac{\ln{k}}{k^2}} \, .
\label{eq:longa}
\end{eqnarray}
The first series, above, reduces to $\,\zeta''(2)$. According to our Eq.~\eqref{eq:somaln}, the last series simplifies to $\,(\pi^2/6)\,\left[ 12\,\ln{A} -\gamma -\ln{(2 \pi)}\right]$. The substitution of these two summations in Eq.~\eqref{eq:longa} completes the proof. \qed
\end{proof}

The result of the above integral evaluation agrees with Eq.~(31) of Ref.~\cite{Blago2015}, as well as Eq.~(6.441.6) of Ref.~\cite{Gradsh}.
\newline

An interesting feature of the Fourier coefficients stated in Eq.~\eqref{eq:bks} arises from the analysis of its asymptotic behavior for large values of $k$. For $\,k=1$, it is clear from Eq.~\eqref{eq:bks} that $\,\eta_1 = \eta$. From that equation it also follows that
\begin{equation}
\eta_{\,2 k} = \eta_k/2 +\ln{2}/(2 \pi k) \, , \quad k \ge 1 \, ,
\label{eq:etas2k}
\end{equation}
and
\begin{equation}
\eta_{\,2^k} = \frac{\:(\ln{2}/\pi) \: k +\eta\,}{2^k} \, , \quad k \ge 0 \, .
\label{eq:etapower2}
\end{equation}
These two formulae were derived by Shamov using the Legendre duplication formula for $\Gamma(x)$, in his answer to a question by Silagadze in Ref.~\cite{Shamov}. He indeed takes Eq.~\eqref{eq:etapower2} into account for establishing the asymptotic behavior of $\,\eta_k\,$ for $\,k \rightarrow \infty\,$ and uses it to develop a heuristic proof of our Lemma~\ref{lema:eta}. His reasoning is so aesthetic that it deserves to be mentioned here. He argues that the integrand in the definition of $\,\eta_k\,$ has only one singularity at $\,t=0$, where $\,\Gamma(t) = -\ln{t} -\gamma \,t +\ldots$, so the asymptotic behavior of the Fourier coefficient $\,\eta_k\,$ should be the same as that of $\,\ln{t}$, i.e. $2 \int_0^1{\sin{(2 \pi k t)}\,\ln{\!\left(t^{-1}\right)} \: d t} = \ln{k}/(\pi k) +\eta/k +\ldots$ He then notes that this result can be written as
\begin{eqnarray*}
-\,\frac{1}{\,\pi k} \, \int_0^{\,2 \pi k}{\frac{\cos{z} -1}{z} ~ dz} = \frac{\,\gamma +\ln{(2 \pi \,k)} -\mathrm{Ci}(2 \pi k)}{\pi \, k} \, ,
\end{eqnarray*}
where $\,\mathrm{Ci}(x)\,$ is the cosine integral, which behaves as $\,(\sin{x})/x\,$ for $\,x \rightarrow \infty$. A comparison of the terms of order $\,1/k\,$ in these expansions promptly yields $\,\eta = [\gamma +\ln{(2 \pi)}]/\pi$, which agrees to our Lemma~\ref{lema:eta}.

\begin{acknowledgements}
The author thanks M.~R.~Javier for checking all closed-form expressions proposed in this work with mathematical software to a high numerical precision.
\end{acknowledgements}

\section*{Conflict of interest}
The author declares that he has no conflict of interest.


\begin{thebibliography}{20}

\bibitem{Adamchik} V.~S.~Adamchik, Polygamma functions of negative order, {\it J. Comp. Appl. Math.} {\bf 100}, 191--199 (1998).

\bibitem{Blago2014} I.~V.~Blagouchine, Rediscovery of Malmsten's integrals, their evaluation by contour integration methods and some related results, {\it Ramanujan J.} {\bf 35}, 21--110 (2014).

\bibitem{Blago2015} I.~V.~Blagouchine, A theorem for the closed-form evaluation of the first generalized Stieltjes constant at rational arguments and some related summations, {\it J. Number Theory} {\bf 148}, 537--592 (2015).

\bibitem{Coffey2011} M.~W.~Coffey, On representations and differences of Stieltjes coefficients, and other relations, {\it Rocky Mountain J. Math.} {\bf 41}, 1815--1846 (2011).

\bibitem{bk:Cohen} H.~Cohen, \textsl{Number theory, vol.~II: Analytic and modern tools}, Springer, New York, 2007.

\bibitem{Connon2009} D.~F.~Connon, Fourier series representations of the logarithms of the Euler gamma function and the Barnes multiple gamma functions. arXiv: 0903.4323 (2009).

\bibitem{Connon2012} D.~F.~Connon, On an integral involving the digamma function. arXiv: 1212.1432 (2012).

\bibitem{Farhi} B.~Farhi, A curious formula related to the Euler Gamma function. arXiv: 1312.7115 (2013).

\bibitem{Glaisher} J.~W.~L.~Glaisher, On the Constant which Occurs in the Formula for $1^1 . 2^2 . 3^3 \ldots n^n$, {\it Messenger Math.} {\bf 24}, 1--16 (1894).

\bibitem{Gradsh} I.~S.~Gradshteyn and I.~M.~Ryzhik, \textsl{Table of Integrals, Series, and Products}, 8th ed. (Academic Press, New York, 2015).

\bibitem{Shamov} \emph{MathOverflow} website, available at: \begin{verbatim} http://mathoverflow.net/questions/158022 \end{verbatim}

\bibitem{Moll} L.~A.~Medina and V.~H.~Moll, A class of logarithmic integrals, {\it Ramanujan J.} {\bf 20}, 91--126 (2009).

\bibitem{Miller} J.~Miller and V.~S.~Adamchik, Derivatives of the Hurwitz Zeta function for rational arguments, {\it J. Comp. Appl. Math.} {\bf 100}, 201--206 (1998).

\bibitem{bk:Rudin} W.~Rudin, \textsl{Principles of Mathematical Analysis}, 3rd ed., McGraw-Hill, New York, 1976.

\end{thebibliography}
\end{document}